\documentclass[12pt]{article}
\usepackage{a4wide}

\usepackage{stmaryrd}
\usepackage[a4paper, margin=2.3cm]{geometry}
\usepackage{amsmath,amssymb,amsthm}
\usepackage{color}
\usepackage{parskip}
\usepackage{hyperref}
\usepackage{parskip}
\usepackage{setspace}
\usepackage{graphicx}
\numberwithin{equation}{section}

\theoremstyle{plain}
\newtheorem{theorem}{Theorem}[section]
\newtheorem{lemma}[theorem]{Lemma}

\theoremstyle{definition}
\newtheorem{definition}[theorem]{Definition}
\newtheorem{notation}[theorem]{Notation}

\theoremstyle{remark}

\newtheorem{case[theorem]}{Case}

\begin{document}
\title{A weak type estimate for regular fractional\\ sparse operators}

\author{
Ji Li\thanks{ Department of Mathematics and Statistics, Macquarie University. Email: {\tt ji.li@mq.edu.au}}
\and 
Chung-Wei Liang\thanks{Department of Mathematics, MD of NTU. Email: {\tt chongweiliang1228@gmail.com}}
\and
Chun-Yen Shen \thanks{Department of Mathematics,  National Taiwan University and NCTS. Email: {\tt cyshen@math.ntu.edu.tw}}
}

\maketitle  
\date{}

\begin{abstract}
In this note the weak type estimates for fractional integrals are studied. More precisely, we adapt the arguments of Domingo-Salazar,  Lacey, and  Rey to obtain improvements for the end point weak type estimates for regular fractional sparse operators. 

\end{abstract}
\section{Introduction}
Let $M$ denote the Hardy--Littlewood maximal function. The Muckenhoupt--Wheeden Conjecture stems from vector-valued estimates for the maximal function. In 1971, C. Fefferman and E. Stein \cite{FS} showed that
\begin{align}
	   \|Mf\|_{L^{1,\infty}(w)} \lesssim \int_{\mathbb{R}^n} |f| M w (x)dx, \label{eq:0.0}
\end{align}
or equivalently
\[
L^1(Mw)\xrightarrow{\quad \text{M} \quad} L^{1,\infty}(\omega),
\]
for arbitrary weight $w$. Throughout this note a weight $w$ is a nonnegative locally integrable function. The above relation gave rise to the following natural question, formulated by Muckenhoupt and Wheeden in $70's$ : Suppose that T is a  Calder\'on--Zygmund singular integral operator, can we have the same mapping property if we replace $M$ in $\eqref{eq:0.0}$ by $T$, that is 
\begin{align}
	    L^1(Mw)\underset{?}{\xrightarrow{\quad \text{T} \quad}} L^{1,\infty}(w). 
\end{align} \\
This problem, known as the \textbf{Muckenhoupt--Wheeden Conjecture}, remained puzzled for quite some time, and was believed to be true. In addition, during the development of this open problem, C. P\'erez \cite{Per} verified that
\begin{align}
	    L^1(M_{L(logL)^{\varepsilon}}w) \xrightarrow{\quad \text{T} \quad} L^{1,\infty}(w).  \label{eq:0.1}
\end{align} 
In particular, if $M^2$ denotes the second iteration of M, then 
\begin{align}
	    L^1(M^2\omega) \xrightarrow{\quad \text{T} \quad} L^{1,\infty}(\omega).  \notag
\end{align} 
We remark here that $ M_{L(logL)^{\varepsilon}}$ is a smaller object than $M^2$ pointwise. Although there are so many remarkable progresses on this problem, no one can achieve the origin estimate. The breakthrough that surprised people is that the Muckenhoupt--Wheeden Conjecture was finally shown to be false in a series of works of M. Reguera and C. Thiele \cite{Har}, \cite{Hil}. The question we are concerned in this note is based on the point of view of $\eqref{eq:0.1}$ for
fractional integral operators $I_\alpha$. It was C. P\'erez \cite{Car} who proved that 
\begin{align}
    L^1(M_\alpha(M_{L(\log L)^\delta}w))\xrightarrow{\quad I_\alpha \quad} L^{1,\infty}(w) ,\,\forall \delta\in (0,1),\label{eq:1.3.1}
\end{align}
for general weight $w$. In the manner of logarithm scale estimate,  Domingo-Salazar,  Lacey, and  Rey \cite{Lac} obtained a better estimate. They imposed an assumption to the Young function $\phi$ so that for Calder\'on--Zygmund operator $T$, any weight $w$ on $\mathbb{R}^d$, it holds that
\[
\underset{\lambda>0}{\sup}\lambda w\left(T^{*}f>\lambda\right)\lesssim c_\phi\int_{\mathbb{R}^d}|f(x)|M_{\phi(L)}w(x)dx.
\]
The proofs in their work can be extended to other dyadic operators which bring us to have the following observation.

To begin with, let $\varphi:[0,\infty)\rightarrow [0,\infty)$ be a \emph{Young function}, that is, a convex, increasing function such that $\varphi(0)=0$ and $\lim_{t\rightarrow\infty}\varphi(t)=\infty$. From these properties, on can deduce that its inverse $\varphi^{-1}$ exists on $(0,\infty)$. Moreover, given a Young function, we can define its complementary function $\psi$ by
\begin{equation*}
\psi(s)=\sup_{t>0}\,\{st-\varphi(t)\}.
\end{equation*}
We will assume that $\lim_{t\rightarrow \infty} \varphi(t)/t=\infty$ to ensure that $\psi$ is finite valued. Under these conditions, $\psi$ is also a Young function and it is associated with the dual space of $\varphi(L)$.

We now state our main result.
\begin{theorem}
Let $\phi$ be any Young function satisfying
\[
c_\phi:=\sum_{k=1}^\infty \frac{1}{\psi^{-1}(2^{2^k})}<\infty,
\]
where $\psi$ is the complementary function of $\phi$.
Then, for any $N$-regular sparse family of cubes $\mathcal{S}$ and any weight $w$ on $\mathbb{R}^d$
\begin{align}
    \left\|\mathcal{A}^\mathcal{S}_{\alpha,\nu} f\right\|_{L^{1,\infty}(w)}\lesssim c_\phi\cdot\left\|f\right\|_{L^1(M_\alpha(M_\phi w))},~\forall \nu\geq1.\label{eq:0.2}
\end{align}
\end{theorem}
It is worth to mention that the original open problem regarding $\eqref{eq:1.3.1}$ which was formulated in the work of C. P\'erez \cite{Car} is that whether the following is true:
 
\begin{align}
    L^1(M_\alpha(Mw))\underset{?}{\xrightarrow{\quad I_\alpha \quad}} L^{1,\infty}(w) ,\,\forall \delta\in (0,1).\label{eq:1.3.2}
\end{align}
It has been proved in \cite{Car} that 
\begin{align}
    L^1(M_\alpha(M_{L(\log L)^\delta}w))\xrightarrow{\quad I_\alpha \quad} L^{1,\infty}(w) ,\,\forall \delta\in (0,1).\notag
\end{align}
However,
\[
{L^1(M_{\alpha}w) \overset{{ I_\alpha }}{\longarrownot\longrightarrow} L^{1,\infty}(w)},
\]
for general weight $w$. Observe that
\[
M_\alpha(M_{L(\log L)^\delta}w)\geq M_\alpha(M_Lw)=M_\alpha(Mw).
\]
In summary, for the Calder\'on--Zygmund operator $T$ and the fractional integral operator $I_\alpha$, what we have known is that
\[
\begin{cases}
L^1(M_\alpha(M_{L(\log L)^\delta}w))\xrightarrow{\quad I_\alpha \quad} L^{1,\infty}(w), \\
L^1(M_{L(\log L)^\delta}w)\xrightarrow{\quad T \quad} L^{1,\infty}(w) 
\end{cases}
\]
and
\[
 L^1(Mw)\overset{\text{T}}{\longarrownot\longrightarrow}  L^{1,\infty}(w).
\]
Our main result states that we can have smaller logarithm scales for any $N$-regular fractional sparse operators.

\section{Dyadic grid and sparse operators}
\begin{definition}[Dyadic Grid]~\\
A dyadic grid, denoted $\mathcal{D}$, is a collection of cubes in $\mathbb{R}^d$ with the following properties:
\begin{itemize}
\item for any $Q \in \mathcal{D} $, there is $k \in \mathbb{Z}$ such that $|Q|=2^{kd}$;
\item if $Q$,$P$ $\in$ $\mathcal{D}$, then $Q\cap P$
$=$ $Q$,\,$P$,or $\varnothing$;
\item for each $k \in \mathbb{Z}$, the family $\mathcal{D}_k:=\{Q \in \mathcal{D}: |Q|=2^{kd}\}$ forms a partition of $\mathbb{R}^d$.
\end{itemize}
\end{definition}

\smallskip
\begin{notation}~
\begin{itemize}
    \item Given any subset $\mathcal{S}\subset\mathcal{D}$, $\mathrm{Ch}_\mathcal{S}(Q)\subset\mathcal{S}$ denote the \textbf{immediate successors}/\textbf{children} of $Q\in\mathcal{D}$;
    \item $\mathbb{P}\subset\mathcal{D}$ denotes a \textbf{dyadic} system with its localized counterpart:
    \[
    \mathbb{P}(Q):=\left\{\textbf{Dyadic}\text{ sub-cubes of }Q\right\},\hspace{1.5ex}\text{where}\hspace{1.5ex}Q\in\mathcal{D}.
    \]
\end{itemize}
\end{notation}

\begin{definition}[Fractional Sparse Operator \cite{Urifra}]~\\
The fractional sparse operator is given by
\[
\mathcal{A}^\mathcal{S}_{\alpha,\nu} f(x):=\left(\underset{Q\in\mathcal{S}}{\sum}\langle f\rangle^\nu_{\alpha,Q}\cdot\chi_{Q}(x)\right)^{\frac{1}{\nu}},
\]
and
\[\langle f\rangle_{\alpha,Q}:=|Q|^{\frac{\alpha}{d}-1}\int_Qf(x)dx,
\]
where $\nu>0$, $0\leq\alpha<d$ and $\mathcal{S}$ is a sparse collection \cite{Len} of dyadic cubes in the sense that
\begin{align}
   \exists \,\Lambda_0 \in (0,1)\ s.t.\ \, \large|\bigcup_{P\in{\mathbb{P}(Q)\cap\mathcal{S}}}P\large|\leq \frac{1}{\Lambda_0}|Q|. \label{eq:2.1}
\end{align}
\end{definition}

It has been known that the operator $\mathcal{A}^\mathcal{S}_{\alpha,\nu}$ dominate some classes of classical operators \cite{Len1}. For instance, for \,$\nu=1$\, and \,$\nu=2$\, with \,$\alpha=0$\,, these operators dominate large classes of Calder\'on–Zygmund singular integrals and Littlewood–Paley square functions, respectively. For \,$\nu=1$\, with \,$0<\alpha<d$\,, \,$\mathcal{A}^\mathcal{S}_{\alpha,\nu}$\, dominate the fractional integral operator $I_\alpha$. However in this note we are unable to prove the result for all the fractional sparse operators. We require the dyadic operators to be even more sparse which we use the following definition.

\begin{definition}[N-Regular Sparse]~\\
We say that a family of sparse cubes $\mathcal{S}$ is $N$-regular if 
\[
\exists N\in \mathbb{N}\,\ s.t.\ \, \forall Q\in \mathcal{S}\implies \#Ch_{\mathcal{S}}(Q)\leq N.
\]
\end{definition}

\section{Proof of Theorem 1.1} 
\begin{proof}\,
Suppose that $\mathcal{S}$ is a family of sparse cubes satisfies $\eqref{eq:2.1}$.
Recall that
\begin{equation}
    L^1(M_\alpha w) \xrightarrow{\quad M_\alpha \quad} L^{1,\infty}(w). \label{eq:8.2}
\end{equation}
Note that $\eqref{eq:0.2}$ is equivalent to show that 
\[
w\left(\left\{x\in\mathbb{R}^d:\Lambda_1\leq\mathcal{A}^\mathcal{S}_{\alpha,\nu} f(x)<2\Lambda_1\right\}\right)\lesssim c_\phi\cdot\left\|f\right\|_{L^1(M_\alpha(M_\phi w))}.
\]

Define $\varepsilon:=\left\{\Lambda_1<\mathcal{A}^\mathcal{S}_{\alpha,\nu} f\leq2\Lambda_1\right\}-\left\{Mf>\Lambda_1^{-1}\right\}$.\\
It suffices to check that
\[
w(\varepsilon)\underset{Chebyshev}{\leq}\frac{1}{\Lambda_1}\int_\varepsilon\mathcal{A}^\mathcal{S}_{\alpha,\nu} f(x)w(x)dx\lesssim\int_{\mathbb{R}^d}|f(x)|M_\alpha(M_{\phi(L)}w)(x)dx.
\]
By getting rid of the set $\left\{M_\alpha f>\Lambda_1^{-1}\right\}$,\, we can eliminate from $\mathcal{S}$ all those cubes $Q$ such that $\langle f\rangle_{\alpha,Q}>\Lambda_1^{-1}$.\\
For $k\in\mathbb{N}$,\,define the set
\[
\mathcal{S}_k:=\{Q\in\mathcal{S}:\Lambda_1^{-k-1}<\langle f\rangle_{\alpha,Q}\leq \Lambda_1^{-k}\},
\]
and set
\[ 
\mathcal{A}^{\mathcal{S}_k}_{\alpha,\nu} f(x):=\left(\sum_{Q\in\mathcal{S}_k}\langle f\rangle^\nu_{\alpha,Q}\cdot\chi_{Q}(x)\right)^{\frac{1}{\nu}}.
\]
The critical lemma is the following:
\begin{lemma}~\\
For each $k\in\mathbb{N}$,\,there is an absolute constant C such that
\[
\int_\varepsilon\mathcal{A}^{\mathcal{S}_k}_{\alpha,\nu} f(x)w(x)dx\leq \frac{2^k}{\Lambda_1^k} w(\varepsilon)+\frac{C}{\psi^{-1}(2^{2^k})}\cdot\int_{\mathbb{R}^d}|f(x)|M_\alpha(M_{\phi(L)}w)(x)dx.
\]
\end{lemma}
\begin{proof}
In general,\,the family of sparse cubes may not have layer structure.\,However,\,we can give it some kind of layer structure.
\begin{itemize}
    \item \textbf{Layer Decomposition}:
\end{itemize}
Write $\mathcal{S}_k$ as the union of $\mathcal{S}_{k,v}$, for $v=0,1,\cdots$, where $\mathcal{S}_{k,0}$ are the maximal elements of $\mathcal{S}_k,$ and $\mathcal{S}_{k,v+1}$ are the maximal elements of $\mathcal{S}_k\setminus\bigcup_{l=0}^v\mathcal{S}_{k,l}.$ We are free to assume that $\mathcal{S}_{k,v}=\varnothing$ if $v>4^{k+1}$.
\begin{itemize}
    \item \textbf{Pairwise Disjoint Dense Subset}:
\end{itemize}
Define that
\[
E_Q:=Q\setminus\bigcup_{P\in Ch_{\mathcal{S}_{k,v+1}}(Q)}P,\,\forall Q\in\mathcal{S}_{k,v}.
\]
Note that $E_Q\overset{disjoint}{\subset}\mathcal{S}_k$.
Set $u:=2^k$. It follows from $\eqref{eq:0.2}$ that for each $v\geq0$, and $Q\in\mathcal{S}_{k,v}$. Then
\[
Q_u:=\underset{{P\in\mathbb{P}(Q)\cap\mathcal{S}_{k,v+u}}}{\bigcup}P\underset{\eqref{eq:0.2}}{\implies} |Q_u|\leq {\Lambda_0}^{-u}|Q|.
\]
For each $Q\in\mathcal{S}_{k,v}$, we decompose the set $\varepsilon\cap Q$ into
\[
\varepsilon\cap Q=\varepsilon\cap\left(\underset{Bottom~ Part}{Q_u}\cup\underset{Layer~Part}{\bigcup_{l=0}^{u-1}\underset{{Q^\prime\in\mathbb{P}(Q)\cap\mathcal{S}_{k,v+l}}}{\bigcup} E_{Q^\prime}}\right),
\]
and hence
\begin{align}
\int_\varepsilon\mathcal{A}^{\mathcal{S}_k}_{\alpha,\nu} f(x)w(x)dx 
\underset{if~\nu\geq1}{\leq}\sum_{v=0}^{\Lambda_1^{k+2}}\sum_{Q\in\mathcal{S}_{k,v}}\langle f\rangle_{\alpha,Q}\cdot w(\varepsilon\cap Q).
\end{align}
We estimate (3.2) by spilt it into \textbf{Layer Part} and \textbf{Bottom Part}.
\begin{itemize}
    \item\textbf{Layer Part}:
\end{itemize}
\begin{align}
 \underset{\textbf{Layer~~Part}\,\,\,\,}{\sum_{v=0}^{\Lambda_1^{k+2}}\sum_{Q\in\mathcal{S}_{k,v}}\sum_{l=0}^{u-1}\sum_{Q^\prime\in\mathbb{P}(Q)\cap\mathcal{S}_{k,v+l}}\langle f\rangle_{\alpha,Q}\cdot w(\varepsilon\cap E_{Q^\prime}) }
 &\leq \Lambda_1^{-k}u\sum_{v=0}^{\Lambda_1^{k+2}}\sum_{Q\in\mathcal{S}_{k,v}}w(\varepsilon\cap E_{Q})\notag\\
 &\leq\frac{2^k}{\Lambda_1^k}w(\varepsilon).\notag
\end{align}\\

\begin{itemize}
    \item\textbf{Bottom Part}:
\end{itemize}
\begin{align}
    \forall Q\in S_{k,v}\implies \langle f\chi_{E_Q}\rangle_{\alpha,Q}
    &=\langle f\rangle_{\alpha,Q}-\underset{P\in Ch_{\mathcal{S}_{k,v+1}}(Q)}{\sum}\left(\frac{|P|}{|Q|}\right)^{1-\frac{\alpha}{d}}\langle f\rangle_{\alpha,P}\notag\\
    &"\gtrsim"\langle f\rangle_{\alpha,Q}-\Lambda_1^{-k}\left(\underset{P\in Ch_{\mathcal{S}_{k,v+1}}(Q)}{\sum}\frac{|P|}{|Q|}\right)^{1-\frac{\alpha}{d}}\notag\\
    &\geq\langle f\rangle_{\alpha,Q}-\Lambda_1^{-k}\Lambda_0^{1-\frac{\alpha}{d}}>\left(1-\Lambda_1\Lambda_0^{1-\frac{\alpha}{d}}\right)\langle f\rangle_{\alpha,Q}.\notag
\end{align}
We note that $"\gtrsim"$ holds if $\mathcal{S}$ is $N$ regular.
\begin{align}
  \implies\underset{\textbf{Bottom~~Part}}{\sum_{v=0}^{\Lambda_1^{k+2}}\sum_{Q\in\mathcal{S}_{k,v}}\langle f\rangle_{\alpha,Q}\cdot w(\varepsilon\cap Q_u)}
 &\lesssim \sum_{v=0}^{\Lambda_1^{k+2}}\sum_{Q\in\mathcal{S}_{k,v}}\int_{E_Q}f(y)dy\cdot \langle w(\varepsilon\cap Q_u)\rangle_{\alpha,Q}\notag \\
 &\lesssim \sum_{v=0}^{\Lambda_1^{k+2}}\sum_{Q\in\mathcal{S}_{k,v}}\frac{1}{\psi^{-1}(2^{2^k})}\cdot\int_{E_Q}f(x)M_\alpha(M_{\phi(L)}w)(x)dx\notag\\
 &\leq\frac{1}{\psi^{-1}(2^{2^k})}\cdot\int_{\mathbb{R}^d}f(x)M_\alpha(M_{\phi(L)}w)(x)dx.\notag
\end{align}

Thus, the proof of this lemma is complete.
\end{proof}
Now, let us finish the main theorem:
\begin{align}
 w(\varepsilon)&\leq\frac{1}{\Lambda_1}\int_\varepsilon\mathcal{A}^{\mathcal{S}}_{\alpha,\nu} f(x)w(x)dx
 =\frac{1}{\Lambda_1}\sum_{k=1}^\infty\int_\varepsilon\mathcal{A}^{\mathcal{S}_k}_{\alpha,\nu} f(x)w(x)dx\notag\\
 &\leq\frac{1}{\Lambda_1}\sum_{k=1}^\infty\left(\frac{2^k}{\Lambda_1^k}w(\varepsilon)+\frac{C}{\psi^{-1}(2^{2^k})}\cdot\int_{\mathbb{R}^d}|f(x)|\cdot M_\alpha(M_{\phi(L)}w)(x)dx\right),\notag
\end{align}
which implies that
\[
w(\varepsilon)\lesssim c_\phi\cdot\int_{\mathbb{R}^d}|f(x)|\cdot M_\alpha(M_{\phi(L)}w)(x)dx,~\forall r\geq 1,
\]
provided that 
\begin{align}
\Lambda_1>2, \,~\,
\frac{1}{\Lambda_1}\cdot\sum_{k=1}^\infty\frac{2^k}{\Lambda_1^k}<1~\text{and}~\Lambda_1\cdot\Lambda_0^{1-\frac{\alpha}{d}}<1.\label{eq:8.3}
\end{align}
Moreover, we are free to assume that $\Lambda_0$ is small enough such that $\eqref{eq:8.3}$ holds for suitable $\Lambda_1$.
\end{proof}

\end{document}